\documentclass{article}
\pdfoutput=1
\usepackage{graphicx,url}
\usepackage[all]{xy}

\usepackage{amsfonts}
\usepackage{amssymb,amsmath,amscd}

\newtheorem{theorem}{Theorem}
\newtheorem{lemma}[theorem]{Lemma}
\newtheorem{proposition}[theorem]{Proposition}
\newtheorem{corollary}[theorem]{Corollary}
\newtheorem{definition}{Definition}

\usepackage{amsfonts,amssymb,amsmath,,amscd}
\newcommand{\C}{\ensuremath{\mathbb{C}}}

\newcommand{\R}{\ensuremath{\mathbb{R}}}

\newcommand{\Z}{\ensuremath{\mathbb{Z}}}

\date{\today}
\begin{document}

\title{The angular momentum of a relative equilibrium}

\author{Alain Chenciner\\
  \\
  \small Observatoire de Paris, IMCCE (UMR 8028), ASD\\
  \small 77, avenue Denfert-Rochereau, 75014 Paris, France\\
  \small \texttt{chenciner@imcce.fr}}

\maketitle

\abstract
There are two main reasons why relative equilibria of $N$ point masses under the influence of Newton attraction are mathematically more interesting to study when space dimension is at least 4: 

1) in a higher dimensional space, a relative equilibrium is determined not only by the initial configuration but also by the choice of a hermitian structure on the space where the motion takes place (see \cite{AC}); in particu\-lar, its angular momentum depends on this choice;

2) relative equilibria are not necessarily periodic: if the configuration is {\it balanced} but not central (see \cite{AC,A2,C1}), the motion is in general quasi-periodic. 

In this exploratory paper we address the following question, which touches both aspects: what are the possible frequencies of the angular momentum of a given central (or balanced) configuration and at what values of these frequencies bifurcations from periodic to quasi-periodic relative equilibria do occur ? We give a full answer for relative equilibrium motions in $R^4$ and conjecture that an analogous situation holds true for higher dimensions. A refinement of Horn's problem given in \cite{FFLP} plays an important role.

\section{A brief review of the relative equilibrium solutions of the $N$-body problem }
Let $x=(\vec r_1,\vec r_2,\cdots,\vec r_N)\in E^N$ be a configuration of $N$ point masses $m_1,m_2,\cdots,m_N$, in the euclidean space $(E,\epsilon)$ of dimension $d$. 
{\it Choosing once and for all an appropriate galilean frame, we shall only consider configurations whose center of mass is fixed at the origin:} $\sum_{k=1}^Nm_k\vec r_k=0.$

\noindent We shall identify $x$ with the $d\times N$ matrix $X$ whose $k$th column is composed of the coordinates $(r_{1k},\cdots,r_{dk})$ of $\vec r_k$ in some orthonormal basis of $E$. 
\smallskip

\noindent{\bf Remark.} Such a matrix $X$ can be thought of as representing the element $x:\xi\mapsto\sum_{k=1}^N\xi_k\vec r_k \in Hom({\cal D}^*,E)\equiv\mathcal{D}\otimes E$, where the {\it dispositions space} $\mathcal{D}$ and its dual ${\cal D}^*$,
$${\cal D}={\R}^N/(1,\cdots,1)\R,\quad {\cal D}^*=\left\{(\xi_1,\cdots,\xi_N)\in\R^N,\sum\xi_i=0\right\},$$ 
were introduced in \cite{AC}. Indeed, considered as a linear mapping from $(\R^N)^*\equiv \R^N$ to $E$, it is the unique extension of $x:{\mathcal D}^*\to E$ which vanishes on the line generated by $(m_1,\cdots,m_N)$.
\smallskip

\noindent The equations of the $N$-body problem may be written $$\ddot x=\nabla U(x),\quad \hbox{where}\quad  U(x)=\sum_{i<j}\frac{m_im_j}{||\vec r_i-\vec r_j||_\epsilon},$$
the gradient being the one defined by the {\it mass scalar product} which, on configurations whose center of mass is at the origin, is given by the formula
$$x'\cdot x''=(\vec r'_1,\vec r'_2,\cdots,\vec r'_N)\cdot (\vec r''_1,\vec r''_2,\cdots,\vec r''_N)=\sum_{k=1}^Nm_k<\vec r'_k,\vec r''_k>_\epsilon,$$

A relative equilibrium solution is an equilibrium of the ``reduced" equations, obtained by going to the quotient by translations (this was already accomplished by choosing a galilean frame where the center of mass is fixed at the origin) and linear isometries. It is proved in \cite{AC} that these are exactly the {\it rigid motions}, where every mutual distance $||\vec r_i-\vec r_j||_\epsilon$ stays constant, that is where the $N$-body configuration behaves as a rigid body. Moreover, the motion is of the form $X(t)=e^{t\Omega}X_0$, that is
$x(t)=(e^{\Omega t}\vec r_1,e^{\Omega t}\vec r_2,\cdots,e^{\Omega t}\vec r_N)$ if $x_0=(\vec r_1,\vec r_2,\cdots,\vec r_N)$, where $\Omega$ is an 
$\epsilon$-antisymmetric operator on $E$ and, {\it if we call $E$ the actual space of motion} (forgetting the non visited dimensions), $\Omega$ is non degenerate. In particular, the dimension of $E$ is even: $d=2p$.
Choosing an orthonormal basis where $\Omega$ is normalized, this amounts to saying that
there exists a hermitian structure on the space $E$ and an orthogonal decomposition $E\equiv\C^p=\C^{k_1}\times\cdots\times\C^{k_r}$ such that
$$x(t)=(x_1(t),\cdots,x_r(t))=(e^{i\omega_1t}x_1,\cdots,e^{i\omega_rt}x_r),$$
where $x_m$ is the orthogonal projection on $\C^{k_m}$ of the $N$-body configuration $x$ and  the action of $e^{i\omega_mt}$ on $x_m$ is the diagonal action on each body of the projected configuration. Such quasi-periodic motions exist only for very special configurations, called {\it balanced configurations} in \cite{AC}, which are characterized by the existence of an $\epsilon$-symmetric endomorphism $\Sigma$ of $E$
such that $\nabla U(x)=\Sigma x$. The most degenerate balanced configurations, indeed the only ones occuring in a space of dimension at most 3 (and hence with the actual space of motion $E$ of dimension 2), are the {\it central configurations} for which $\Sigma=\lambda Id$. 
In this case, $\Omega=\omega J$, with $J$ a hermitian structure on $E$, and the motion is
$$x(t)=(\vec r_1(t),\cdots,\vec r_N(t))=e^{i\omega t}x_0=(e^{i\omega t}\vec r_1,\cdots,e^{i\omega t}\vec r_N)$$ in the hermitian space $E\equiv\C^{2p}$; in particular, it is periodic.

\section{Angular momentum}
Given a configuration $x=(\vec r_1,\cdots,\vec r_N)$ and a configuration of velocities $y=\dot x=(\vec v_1,\cdots, \vec v_N)$, both with center of mass at the origin: 
$\sum_{k=1}^Nm_k\vec r_k=\sum_{k=1}^Nm_k\vec v_k=0$, the {\it angular momentum} of $(x,y)$ is the bivector {\it (we use the french convention $^{t\!}x$ for the transposed of $x$)}
$${\mathcal C}=\sum_{k=1}^Nm_k\vec r_k\wedge\vec v_k=-x\mu^{t\!}y+y\mu^{t\!}x\in\wedge^2E\equiv Hom_a(E^*,E)$$
(the isomorphism $\mu:{\mathcal D}\to{\mathcal D}^*$ is the mass scalar product on ${\mathcal D}$, see\cite{AC}).
In an orthonormal basis of $E$ where $x$ and $y$ are repectively represented by the $d\times N$ matrices $X$ and $Y$,  setting $M=diag(m_1,\cdots,m_N)$, it can be identified with the antisymmetric matrix
$$C=-XM^{t\!}Y+YM^{t\!\!}X\;\;\hbox{with coefficients}\;\; c_{ij}
=\sum_{k=1}^Nm_k(-r_{ik}v_{jk}+r_{jk}v_{ik}).$$
The dynamics of a solid body is determined by its {\it inertia tensor} (with respect to its center of mass) 
$${\mathcal S}=x\mu^{t\!}x\in Hom_s(E^*,E).$$ 
In the case of an 
$N$-body configuration $x$ whose center of mass is at the origin, ${\mathcal S}$ is identified with the symmetric matrix 
$$S=XM^{t\!\!}X\;\;\hbox{with coefficients}\;\; 
s_{ij}=\sum_{k=1}^Nm_kr_{ik}r_{jk},$$
whose trace is the {\it moment of inertia of the configuration $x$ with respect to its center of mass}, that is its square norm $|x|^2$ for the mass scalar product: 
$$trace\, S=I(x)=|x|^2.$$
In particular, the angular momentum of a relative equilibrium solution $x(t)=e^{t\Omega}x_0$
is represented, if $S_0=X_0M^{t\!\!}X_0$, by the antisymmetric matrix 
$$C=S_0\Omega+\Omega S_0.$$

\begin{lemma} The symmetric matrices $\Omega^2$ and $S_0$ commute.
\end{lemma}
\begin{proof}
The invariance under isometries of the Newton potential implies (see \cite{AC} but you can do it by yourself) the factorization
$\nabla U(x)=2xA$, where the $\mu^{-1}$-symmetric {\it Wintner-Conley endomorphism} $A:{\mathcal D}^*\to{\mathcal D}^*$ is represented by the $N\times N$ matrix 
$$A=\frac{1}{2}\begin{pmatrix}
-\sum_{i\not=1}\frac{m_i}{||\vec r_1-\vec r_i||^3}&\frac{m_1}{||\vec r_1-\vec r_2||^3}&\cdots&\frac{m_1}{||\vec r_1-\vec r_N||^3}\\
\frac{m_2}{||\vec r_2-\vec r_1||^3}&-\sum_{i\not=2}\frac{m_i}{||\vec r_2-\vec r_i||^3}&\cdots&\frac{m_2}{||\vec r_2-\vec r_N||^3}\\
\cdots&\cdots\cdots&\cdots&\cdots\\
\frac{m_N}{||\vec r_N-\vec r_1||^3}&\frac{m_N}{||\vec r_N-\vec r_2||^3}&\cdots&-\sum_{i\not=N}\frac{m_i}{||\vec r_N-\vec r_i||^3}
\end{pmatrix}$$
which satisfies $AM=M^{t\!\!}A$.  As $\ddot X(0)=\Omega^2X_0=2X_0A$, this implies that
$\Omega^2S_0=2X_0AM^{t\!\!}X_0$ is symmetric, which proves the lemma (see also \cite{BCRT}). 
\end{proof}
\smallskip

\noindent {\bf Remark.} Doing the same computation on the $\mathcal{D}$ side instead of the $E$ side leads to the equation of balanced configurations (see \cite{AC}):
$$^t\!X_0X_0A=\frac{1}{2}{^t\!X_0}\Omega^2X_0=A^t\!X_0X_0,$$ that is the commutation of $A$ with the {\it intrinsic inertia} $B={^t\!X_0}X_0$.

\goodbreak

\begin{corollary} Let $x_0$ be an N-body balanced configuration in the euclidean space $(E,\epsilon)$. Let $x(t)=e^{t\Omega }x_0$ be a relative equilibrium motion  of $x_0$. The frequencies of its angular momentum $\mathcal{C}\in\Lambda^2E$ (i.e. the moduli of the eigenvalues of the antisymmetric matrix $C=S_0\Omega+\Omega S_0$ can be written  
$$\nu_1^1,\cdots,\nu_1^{k_1+1},\nu_2^1,\cdots,\nu_2^{k_2+1},\cdots,\nu_r^1,\cdots,\nu_r^{k_r+1}$$ so that the following identity holds true:
$$\frac{1}{\omega_1}(\nu_1^1+\cdots+\nu_1^{k_1+1})+\frac{1}{\omega_2}(\nu_2^1+\cdots+\nu_2^{k_2+1})+\cdots+\frac{1}{\omega_r}(\nu_r^1+\cdots+\nu_r^{k_r+1})=I(x_0).$$
\end{corollary}

\begin{proof} One chooses an orthonormal basis of $E$ in which the matrices 
$\Omega^2$ and $S_0$ are both diagonal:
\begin{equation*}
\left\{
\begin{split}
\Omega^2&=diag(-\omega_1^2,-\omega_1^2,\cdots,-\omega_2^2,-\omega_2^2,\cdots,
-\omega_r^2,-\omega_r^2),\\
S_0&=diag(\sigma_1,\sigma_2,\cdots,\sigma_{2p}).
\end{split}
\right.
\end{equation*}
To such a basis is associated an orthogonal decomposition 
$$E=E_1\oplus E_2\oplus\cdots\oplus E_r$$
 of $E$ into blocks of dimensions $2(k_1+1),2(k_2+1),\cdots,2(k_r+1)$, on each of which 
 $\Omega$ is a multiple of a complex structure:
$$\Omega=diag(\omega_1J_1,\omega_2J_2,\cdots,\omega_rJ_r),$$
where $J_i$ is a hermitian structure on $E_i\equiv\R^{2k_i}$. 
It follows that $C$ itself decomposes into $r$ blocks of  the form $\omega_i(s_iJ_i+J_is_i)$,
where $s_1,\cdots, s_r,$ are the diagonal blocks of $S_0$. 
Hence we are reduced to proving the trace identity in each block and adding them to get $I(x_0)=trace S_0=\sum_{i=1}^rtrace (s_i)$.

This (trivial) fact will be established in the next section where we study the case when 
$\Omega=\omega J$ is a multiple of a hermitian structure.
\end{proof}
\smallskip

\noindent {\bf Remark.} In his thesis \cite{A1}, Alain Albouy proved that the rank $2k$ of the angular momentum of an $n$-body motion satisfies the following inequalities, generalizing Dziobek's theorem:
$$2k\le d\le k+n-1,$$
where $d$ is the dimension of the actual space of motion, that is the dimension of the subspace generated at any instant by the bodies and their velocities. Recall that in the case of a relative equilibrium, $d$ is necessarily even. 
\smallskip

\section{The frequency mapping of a central configuration}

Given some inertia $2p\times 2p$ matrix $S_0$ (say, the one of a central configuration $x_0$, but it could be any symmetric non-negative matrix), we study the mapping $$J\mapsto \omega^{-1}C=S_0J+JS_0$$ from the space of hermitian structures on $E$ to the set of $2p\times 2p$ antisymmetric real matrices. We shall only be interested in the spectra of the matrices $\omega^{-1}C$, hence choosing an orientation for $J$ is harmless and we shall consider only those of the form $J=R^{-1}J_0R$, where $J_0=\begin{pmatrix}0&-Id\\ Id&0\end{pmatrix}$ and $R\in SO(2p)$. 
The matrix $\omega^{-1}C$ is actually $J$-skew-hermitian, with spectrum $i\nu_1,\cdots,i\nu_p$ if considered as a complex matrix. Replacing it by $\Sigma=J_0^{-1}R\omega^{-1}CR^{-1}=J_0^{-1}SJ_0+S$, where $S=RS_0R^{-1}$, makes it $J_0$-hermitian with spectrum 
$\nu_1,\cdots,\nu_p$ that we can suppose to be ordered. 
\begin{definition}
The {\it frequency mapping} 
$${\mathcal F}:U(p)\backslash SO(2p)\to W_p^+=\{(\nu_1,\cdots\nu_p)\in\R^p,   \nu_1\ge\cdots\ge\nu_p\}$$ associates to each hermitian structure $J=R^{-1}J_0R$ the ordered spectrum
$(\nu_1,\cdots,\nu_p)$ of the $J_0$-hermitian matrix
$\Sigma=J_0^{-1}(RS_0R^{-1})J_0+RS_0R^{-1}$.
\end{definition}
\smallskip

\noindent We have chosen the notation $U(p)\backslash SO(2p)$ because two rotations $R,R'$ define the same hermitian structure if and only if $R'=UR$ with $U\in U(p)$.

\noindent The fibers of this mapping correspond to positive hermitian structures $J_1=R_1^{-1}J_0R_1,\, J_2=R_2^{-1}J_0R_2$, which define relative equilibria 
whose angular momenta are conjugated under complex isomorphisms from $(E,J_1)$ to $(E,J_2)$, that is: there exists $U\in U(p)$ such that 
$$R_1C_1R_1^{-1}=UR_2C_2R_2^{-1}U^{-1}.$$

\noindent{\bf Explicit formul\ae.}
Let $R\in SO(2p)$ be a rotation and
$$\vec h_i=R^{-1}(\vec e_{i}),\;\vec k_i=R^{-1}(\vec e_{p+i}),\quad i=1,\cdots,p,$$ 
be the {\it line} vectors of the matrix which represents $R$ in the orthonormal basis $\{\vec e_1,\cdots,\vec e_{2p}\}$ of $E\equiv\R^{2p}$. The coefficient at line $i$ and column $j$ of the Gram matrix $RS_0R^{-1}$ is  $<R^{-1}(\vec e_i), R^{-1}(\vec e_j)>_{S_0}$,  where the scalar product $<. ,. >_{S_0}$ defined by the symmetric matrix $S_0$ is possibly degenerate. 
Hence the real matrix $RCR^{-1}=\omega(J_0RS_0R^{-1}+RS_0R^{-1}J_0)$ is of the form
$$RCR^{-1}=\omega\begin{pmatrix}
A&-B\\B&A
\end{pmatrix},$$ where
$A$ and $B$ are $p\times p$ matrices whose coefficients $a_{ij}$ and $b_{ij}$ are respectively
$$a_{ij}=<\vec h_i,\vec k_j>_{S_0}-<\vec k_i,\vec h_j>_{S_0},\;\; b_{ij}=<\vec h_i,\vec h_j>_{S_0}+<\vec k_i,\vec k_j>_{S_0}.$$
Considered as a $J_0$-skew-hermitian matrix $\omega(A+iB)$, its coefficients are
$$c_{ij}=\omega(a_{ij}+ib_{ij})=i\omega <\vec h_i+i\vec k_i,\overline{\vec h_j+i\vec k_j}>_{S_0}.$$
Hence, the coefficients $\sigma_{ij}$ of the $J_0$-hermitian matrix $\Sigma=B-iA$ are given by
$$\sigma_{ij}=\frac{1}{i\omega}c_{ij}=<\vec h_i+i\vec k_i,\overline{\vec h_j+i\vec k_j}>_{S_0}.$$
\smallskip

\begin{lemma} The trace of the complex hermitian matrix $\Sigma={\mathcal F}(J)$ is equal to the trace of the real symmetric matrix $S_0$:
$$\nu_1+\nu_2+\cdots+\nu_p=\hbox{trace}\; S_0.$$
\end{lemma}
\begin{proof} 
The trace of the real symmetric matrix $\Sigma$ is equal on the one hand to $2\, trace S_0$, on the other hand to  $2(\nu_1+\cdots+\nu_p)$ because  its spectrum is $\nu_1,\nu_1,\nu_2,\nu_2,\cdots,\nu_p,\nu_p$. One can also use the explicit formula.
\end{proof}
\smallskip

{\it In the sequel we fix an orthonormal basis of $E$ in which $S_0$ is diagonal, 
$$S_0=diag(\sigma_1,\sigma_2,\cdots,\sigma_{2p}),$$}
which means that the coordinate axes are in the directions of axes of the inertia ellipsoid of $x_0$:
\smallskip

\noindent {\bf Question.} Understand the image of ${\mathcal F}$.  The question starts being non trivial when $p\ge 2$ as, if $p=1$, $\Lambda^2E\equiv \R$, the hermitian structure is unique up to orientation and any non zero value of the angular momentum is attained with a configuration of the form $\lambda x_0$. 
\smallskip

\noindent{\bf Remarks.} {\bf 1)} Our problem is the one of understanding the periodic motions of a rigid body which moves freely around its center of mass (see \cite{Ar}). Notice that in restriction to this set of periodic motions, the Euler equations become trivial, as they merely state the constancy of the angular momentum in a frame attached to the configuration (the rigid body); equivalently, as $S_0$ commutes with $J^2=-Id$, the angular momentum and the angular velocity commute.
\smallskip

{\bf 2)} The space $U(p)\backslash SO(2p)$ of positive hermitian structures on $(E,\epsilon)$ is of dimension $p(p-1)$. It is the same as the {\it isotropic grassmaniann} and was identified by Elie Cartan with the space of {\it projective pure spinors} (see \cite{LM}). For $p=2$ and $p=3$, it is respectively diffeomorphic to $P_1(\C)$ and $P_3(\C)$. More generally, it is a hermitian symmetric space whose Lusternik-Schnirelman category is $2^{p-1}$. As the adjoint orbit of $J_0$ in the Lie algebra $so(2p)$, it inherits a Kostant-Souriau symplectic form $\omega_J(X,Y)=trace(J[X,Y])$; notice that $U(p)\backslash SO(2p)$ inherits no symplectic structure from its embedding $J\mapsto (x_0,\omega Jx_0)$ into the phase space of the $N$-body problem because its image is contained in the fiber of $x_0$ which is a lagrangian submanifold.

\section{The symmetry group $\Gamma$}

Given two diagonal elements of $O(2p)$,  $$D'=diag(u_1,u_2,\cdots, u_{2p}),\quad  D''=diag(v_1,v_2,\cdots, v_{2p})$$
such that $\det D'\times \det D''=+1$, 
the map
$$R\mapsto f(R)=D'RD''$$
is an involution of $SO(2p)$, which replaces the coefficient $\rho_{ij}$ of a rotation matrix by $u_iv_j\rho_{ij}=\pm \rho_{ij}$.
If moreover we suppose that
$$\exists \eta=\pm 1,\quad D'J_0D'=\eta J_0,$$
which amounts to $u_{i}u_{p+i}=\eta$ for all $i=1,2,\cdots,p$,  
this operation transforms the complex structure $J=R^{-1}J_0R$ into the complex structure $\eta D''JD''$
so that we get a group $\Gamma$ of involutions of $U(p)\backslash SO(2p)$ by choosing the matrices $D'$ and $D''$ of the form
$$D'=\begin{pmatrix}
Id&0\\
0&\eta Id
\end{pmatrix},\quad 
D''=diag(v_1,v_2,\cdots,v_{2p}),\quad \eta^p\det D''=+1,$$
and noticing that the couples $(\eta,D'')$ and $(\eta,-D'')$ define the same involution. The cardinal of $\Gamma$ is one half of the product of 2 (choices for $\eta$) and $2^{2p-1}$ (choices for $D''$ with determinant equal to $\eta^p$), that is
$$\Gamma\equiv \left(\Z/2\Z\right)^{2p-1}.$$

\noindent If, moreover, the orthonormal basis of $E$ which is chosen to identify $E$ with $\R^{2p}$ is such that
the matrix $S_0$ is diagonal, 
we have $D''S_0D''=S_0$. It follows that 
$\Sigma=S_0+J^{-1}S_0J$ is transformed into $D''\Sigma D''$.
and that the mapping $\mathcal{F}=\mathcal{F}_{S_0}$ factorizes through a mapping 
$$\mathcal{F}: U(p)\backslash SO(2p)/\Gamma\ni R\mapsto (\nu_1,\cdots\nu_p)\in W_p^+.$$

\section{Adapted hermitian structures}\label{Adapted} Among positive hermitian structures, a special role is played by structures which are fixed by elements of $\Gamma$ of a special type:
\begin{definition} In an orthonormal basis of $E$ where the inertia matrix $S_0$ is diagonal, an ``adapted hermitian structure" is a positive hermitian structure which can be written (not uniquely) 
$$J_{\rho, P}=P^{-1}\begin{pmatrix}
0&-\rho^{-1}\\
\rho&0
\end{pmatrix}P=R^{-1}J_0R,\quad R=\begin{pmatrix}
\rho&0\\
0&Id
\end{pmatrix}P
,$$
where $\rho\in SO(p)$ and $P\in SO(2p)$ is a signed permutation. 

When $\rho$ can be chosen equal to $Id$, that is when $R$ is a signed permutation, one speaks of a ``basic hermitian structure".
\end{definition} 

\noindent A filtration of the set of adapted structures is obtained by assigning the rotations $\rho$ to belong to the subgroup 
$SO(k_1+1)\times \cdots \times SO(k_r+1)$ of $SO(p)$ formed by the elements which respect the factors of a direct sum decomposition 
$\R^p=\R^{k_1+1}\oplus\cdots\oplus\R^{k_r+1}$ where each factor is generated by elements of the canonical basis of $\R^p$. We shall speak in this case of {\it $(k_1,\cdots,k_r)$-adapted structures.}

\begin{lemma} The signed permutation $P$ being given, the adapted hermitian structures 
$J_{\rho, P}$ are precisely the ones which are fixed by the involution $(-1,D''_P)\in\Gamma$, where
$$D''_P=P\begin{pmatrix}
Id_{p\times p}&0\\
0&-Id_{p\times p}
\end{pmatrix}P^{-1}\; ,$$
The basic hermitian structures are the ones fixed by a subgroup of $\Gamma$ of order $2^p$ consisting of elements $(\eta,D'')$ of the following form: there exists a signed permutation $P$ such that  if the $\tilde v_i$ are the coefficients of the diagonal matrix $PD''P^{-1}$, one has $\eta\tilde v_i\tilde v_{p+i}=+1$ for $i=1,\cdots,p$. 
\end{lemma}
\begin{proof}
The element $(\eta,D''=diag(v_1,v_2,\cdots,v_{2p}))\in \Gamma$ leaves invariant the complex structure $J=(a_{ij})_{1\le i,j\le 2p}$ 
if and only if $\eta D''JD''=J$, that is
$$\forall i,j,\; 1\le i,j\le 2p,\quad  (\eta v_iv_j=-1)\Longrightarrow  (a_{ij}=0),$$
hence the conclusion when $P$ is the Identity. It remains to notice that $J$ is invariant under $(\eta,D''_P)$ if and only if $P^{-1}JP$ is invariant under $(\eta,D''_{Id})$.  
\end{proof}
\smallskip

\noindent More generally, the $(k_1,\cdots, k_r)$-adapted structures are fixed by a subgroup of $\Gamma$ of order $2^r$.

\subsection{Partitions} The (signed) permutation $P$ being fixed, the set of $(k_1,\cdots, k_r)$-adapted structures is determined by a partition of the basis $\{e_1,\cdots,e_{2p}\}$ of $E$ into $2r$ subsets $K_1,L_1,\cdots,K_r,L_r$ such that $K_i$ and $L_i$ are both of cardinal $k_i+1$. Let $E_{K_i}$ (resp. $E_{L_i}$) be the subspace of $E$ generated by the elements of $K_i$ (resp. $L_i$). The $(k_1,\cdots, k_r)$-adapted structures $J$ are the ones which send $E_{K_i}$ onto $E_{L_i}$ for all $i=1,\cdots, r$.

\noindent In particular,  if $P^{-1}(\vec e_i)=\epsilon_i\vec e_{\pi(i)}$, a {\it basic} hermitian stucture is associated with a partition of the set $\{1,2,\cdots,2p\}$ into $p$ pairs 
$$\{K_i,L_i\}=\{\pi(i),\pi(p+i)\},\, i=1,\cdots,p,$$
such that the 2-planes generated by $e_{\pi(i)}, e_{\pi(p+i)}$ are complex lines (there are $\frac{(2p)!}{2^pp!}$ such partitions, that is $1.3.5\cdots (2p-1)$). For example, the basic hermitian structure $J_0$ is associated with the partition $(1,p+1)(2,p+2)\cdots(p,2p)$. 
\smallskip

\noindent On the other hand, the whole manifold of {\it adapted} hermitian structures $J_{\rho,P}$ corresponding to a given permutation $P$ is associated with a partition of the set $\{1,2,\cdots,2p\}$ into two subsets 
$$K_1=\{\pi(1),\cdots,\pi(p)\}\quad\hbox{and}\quad L_1=\{\pi(p+1),\cdots,\pi(2p)\}$$ with $p$ elements, each of these complex structures $J$ sending the $p$-dimensional subspace generated by $e_{\pi(1)},\cdots, e_{\pi(p)}$ onto the orthogonal  subspace generated by $e_{\pi(p+1)},\cdots, e_{\pi(2p)}$. There are $C_{2p}^p=\frac{(2p)!}{p!p!}$ such partitions.
In particular, the adapted hermitian structures of  the form $J_{\rho,Id}$ are associated with the partition $(1,2,\cdots,p)(p+1,p+2,\cdots,2p)$.

\section{The image of the frequency mapping}

\subsection{Adapted structures and Horn problem}

For adapted hermitian structures, we get
$$\Sigma=\begin{pmatrix}
\rho\sigma_-\rho^{-1}+\sigma_+&0\\
0&\rho\sigma_-\rho^{-1}+\sigma_+
\end{pmatrix}\quad\hbox{(real form)},$$
where $\sigma_-$ and $\sigma_+$ are the $p\times p$ diagonal matrices such that
$$PS_0P^{-1}=\begin{pmatrix}
\sigma_-&0\\
0&\sigma_+
\end{pmatrix},$$
that is 
$$\Sigma=\rho\sigma_-\rho^{-1}+\sigma_+\quad\hbox{(complex form)}.$$ 
\smallskip

\noindent If $P^{-1}(\vec e_i)=\epsilon_i\vec e_{\pi(i)}$, that is if the coefficients of $P$ are $p_{ij}=\epsilon_i\delta_{j\pi(i)}$ where $\delta$ is the Kronecker symbol, we have
$$\sigma_-=diag(\sigma_{\pi(1)},\sigma_{\pi(2)},\cdots,\sigma_{\pi(p)}),\;\; \sigma_+=diag(\sigma_{\pi(p+1)},\sigma_{\pi(p+2)},\cdots,\sigma_{\pi(2p)}).$$
In particular, if $J_{Id,P}$ is a basic hermitian structure, the frequencies $\nu_i$ of the associated angular momentum are, up to reordering, the $\sigma_{\pi(i)}+\sigma_{\pi(p+i)}$.
\smallskip

\noindent{\bf Notations.} Let $A_{\sigma_-,\sigma_+}\subset W_p^+$ be the set of ordered spectra of a sum of real symmetric $p\times p$  matrices with spectra 
$\sigma_-$ and $\sigma_+$ respectively (here we identify the diagonal matrices $\sigma_-$ and $\sigma_+$ with the set of their elements, let Bourbaki forgive us).
Let $B_{\sigma_-,\sigma_+}\in W_p^+$ be the ordered spectrum of the diagonal matrix $\sigma_-+\sigma+$.
The images $\mathcal{A}$ and $\mathcal{B}$ under the frequency mapping 
$\mathcal{F}$ of  the adapted  and the basic hermitian structures are respectively 

$$\mathcal{A}=\cup A_{\sigma_-,\sigma_+},\quad \mathcal{B}=\cup B_{\sigma_-,\sigma_+},$$
the union being on all the couples 
$(\sigma_-,\sigma_+)$ such that $\sigma_-\cup\sigma_+=\{\sigma_1,\cdots,\sigma_{2p}\}$.
\smallskip

\noindent Finally, let $\mathcal{A}_{k_1,\cdots ,k_r}$ be the image under $\mathcal{F}$ of the $(k_1,\cdots, k_r)$-adapted hermitian structures (in particular, $\mathcal{A}=\mathcal{A}_{p}$ and $\mathcal{B}=\mathcal{A}_{0,\cdots ,0}$). The tableau which follows  shows the orders of the subgroups $\gamma$ of $\Gamma$, the dimensions of the corresponding invariant subsets of hermitian structures and the generic dimension of their images under $\mathcal{F}$:
$$
\begin{array}[]{|c||c|c|c|c|c||c|c|}\hline
order \;\gamma&2^p&\cdots&2^r&\cdots&2&1\\ \hline
dim.\; inv.&0&\cdots&\frac{(k_1+1)(k_1)}{2}+\cdots+\frac{(k_r+1)(k_r)}{2}&\cdots&\frac{p(p-1)}{2}&p(p-1)\\ \hline
dim. \;image&0&\cdots&p-r&\cdots&p-1&p-1\\
\hline
\end{array}
$$
\smallskip

\begin{proposition} The set $\mathcal{A}$ is a convex polytope. The subsets 
$\mathcal{A}_{k_1k_2\cdots k_r}$ (resp. $\mathcal{B}$) lie in the faces of $\mathcal{A}$ (resp. the vertices of $\mathcal{A}$ located in the interior of $W_p^+$).
\end{proposition}

\begin{proof}
Thanks to the works of Weyl, Horn, Atiyah, Guillemin \& Sternberg, Kirwan, Klyachko, Knutson \& Tao, Fulton,$\hdots$, the subsets $A_{\sigma_-,\sigma_+}$, which can be described as the images of moment maps, are well understood. In particular they are convex polytopes. More precisely, we have:
\begin{theorem}[see\cite{K}] Let $\mathcal{O}_\lambda,\mathcal{O}_\mu$ be the spaces of Hermitian matrices with spectrum $\lambda,\mu$. Let $e:\mathcal{O}_\lambda\times\mathcal{O}_\mu\to\R^n$ take a pair of matrices to the spectrum of their sum, listed in decreasing order. Then the image of $e$ is a convex polytope. Also, if $e(H_\lambda,H_\mu)$ is on a face of the image polytope and is also a strictly decreasing list, then $H_\lambda,H_\mu$ are simultaneously block diagonalisable.
\end{theorem}
We need a version of this theorem which applies to the real symmetric matrices. Such a version exists, at least for the first part asserting that the image is a convex polytope, indeed the same as in the hermitian case (see \cite{F2} Theorem 3). The characterization of the faces, nevertheless, holds only in the sense that the matrices need be simultanously block diagonalisable only in the complex domain (\cite{F2} Theorem 5 and the counter-example which follows the statement of the theorem).
\goodbreak

\smallskip

\noindent It remains to notice that the union $\mathcal{A}$ of the convex polytopes $\mathcal{A}_{\sigma_-,\sigma_+}$ corresponding to the various partitions 
$\{\sigma_1,\cdots,\sigma_{2p}\}=\sigma_-\cup\sigma_+$ forms itself a convex polytope whose extremal points not on the boundary of the Weyl chamber (i.e. corresponding to strictly decreasing sets of eigenvalues) belong to ${\mathcal{B}}$.  Indeed, the following is true:
\begin{theorem}[\cite{FFLP} Proposition 2.2]  Let $A$ and $B$ be $p\times p$ Hermitian matrices. Let  $\sigma_1\ge\sigma_2\ge\cdots\ge\sigma_{2p}$ be the eigenvalues of $A$ and $B$ arranged in descending order. Then there exist Hermitian matrices $\tilde A$ and $\tilde B$ with eigenvalues $\sigma_1\ge\sigma_3\ge\cdots\ge\sigma_{2p-1}$ and $\sigma_2\ge\sigma_4\ge\cdots\ge\sigma_{2p}$ respectively, such that $\tilde A+\tilde B=A+B$.
\end{theorem}

Hence, if $\sigma_1\ge\sigma_2\ge\cdots\ge\sigma_{2p}$, the set $\mathcal{A}$ coincides with $A_{\sigma_-,\sigma_+}$, where 
$$\sigma_-=\{\sigma_1\ge\sigma_3\ge\cdots\ge\sigma_{2p-1}\},\quad \sigma_+=\{\sigma_2\ge\sigma_4\ge\cdots\ge\sigma_{2p}\}.$$ 
\end{proof}
\subsection{Some questions and a conjecture}

\noindent{\bf QUESTION 1}: Is there a big polytope in some product space whose different projections give the various $A_{\sigma_-,\sigma_+}$ ?
In a sense, the conjecture below would give a possible answer.
\smallskip

\noindent{\bf QUESTION 2}: Is the image of $\mathcal{F}$ itself the same as the image of some moment map ? (the answer is probably ``no")
\smallskip

\noindent{\bf QUESTION 3}: What is the union over all the central configurations $x_0$ with 
moment of inertia $I(x_0)=\hbox{trace}\; S_0=1$ of the images of $\mathcal{F}_{S_0}$~? (here, one can either fix the number $N$ of bodies or let it be arbitrary).
\smallskip

\noindent{\bf CONJECTURE}\quad  {\it $Im\mathcal{F}$ coincides with the convex polytope $\mathcal{A}$.} 
\smallskip

\noindent This is equivalent to asserting that every level of $\mathcal{F}$ meets some adapted hermitian structure i.e. to proving that for every such level, the action of at least one of the involutions $(-1,D''_P)$ has a fixed point inside this level.  
\smallskip

In what follows, we prove the conjecture in the simple case $p=2$ (the case $p=1$ is  trivial) and 
give a picture of $\mathcal{A}$ in a case with $p=3$.

\goodbreak

\section{The case of $N$ bodies in $E=\R^4$}

\subsection{The trivial case $E=\R^2$} As $SO(2)=U(1)$, the image of the frequency map is the single point $\sigma_1+\sigma_2$.

\subsection{A parametrization of $SO(4)/U(2)$}
Each rotation $R\in SO(4)$ is equivalent, modulo left multiplication by an element of  $U(2)$, to a rotation of the form

$$R(\varphi,\theta)=\begin{pmatrix}
\cos\varphi&-\sin\varphi\cos\theta&\sin\varphi\sin\theta&0\\
\sin\varphi&\cos\varphi\cos\theta&-\cos\varphi\sin\theta&0\\
0&\sin\theta&\cos\theta&0\\
0&0&0&1
\end{pmatrix}.
$$
Indeed, there exists a unique element $U\in SU(2)$ such that $UR(\vec e_4)=\vec e_4$ and a rotation ${\hat U}\in U(2)$ of the $\{\vec e_1,\vec e_3\}$ plane which sends  $UR(\vec e_1)$ on a vector in the plane $\{\vec e_1,\vec e_2\}$. It follows that ${\hat U}UR$ takes the form of $R(\varphi,\theta)$. The mapping 
\begin{equation*}
\begin{split}
R\mapsto -J\vec e_4&=-R^{-1}J_0R\vec e_4
=-R(\varphi,\theta)^{-1}J_0R(\varphi,\theta)\vec e_4
=R(\varphi,\theta)^{-1}\vec e_2\\
&=\vec h_2=(\sin\varphi,\cos\varphi\cos\theta,-\cos\varphi\sin\theta,0)\in S^2\subset\R^3
\end{split}
\end{equation*} 
factors through a bijection from $SO(4)/U(2)$ to the 2-sphere $S^2\subset \R^3$ (compare to McDuff-Salamon). 
The mapping
$(\varphi,\theta)\mapsto R(\varphi,\theta)^{-1}\vec e_3$ is nothing but a version of ``spherical coordinates" on $S^2$ 
which explode at $\varphi=\pi/2\mod\pi$.
The corresponding complex structure $J(\varphi,\theta)=R(\varphi,\theta)^{-1}J_0R(\varphi,\theta)$ is
$$J(\varphi,\theta)=\begin{pmatrix}
0&-\cos\varphi\sin\theta&-\cos\varphi\cos\theta&-\sin\varphi\\
\cos\varphi\sin\theta&0&\sin\varphi&-\cos\varphi\cos\theta\\
\cos\varphi\cos\theta&-\sin\varphi&0&\cos\varphi\sin\theta\\
\sin\varphi&\cos\varphi\cos\theta&-\cos\varphi\sin\theta&0
\end{pmatrix}.
$$

\subsection{The image of $\mathcal{F}$}

{\it As above, we choose an orthonormal basis of $(E,\epsilon)$ such that} 
$$S_0=diag(\sigma_1,\sigma_2,\sigma_3,\sigma_4).$$
The hermitian endomorphism $\Sigma:\C^2\to\C^2$ is represented by the complex  $2\times 2$ matrix
$$\begin{pmatrix}
c_{11}&c_{12}\\
c_{21}&c_{22}
\end{pmatrix}=
\begin{pmatrix}
(||\vec h_1||^2_{S_0}+||\vec k_1||^2_{S_0})&<\vec h_1+i\vec k_1,\overline{\vec h_2+i\vec k_2}>_{S_0}\\
<\vec h_2+i\vec k_2,\overline{\vec h_1+i\vec k_1}>_{S_0}&(||\vec h_2||^2_{S_0}+||\vec k_2||^2_{S_0})
\end{pmatrix},
$$
hence
$$\nu_1,\nu_2=\left(\frac{1}{2}I(x_0)\pm \sqrt{\delta}\right),$$ with
\begin{equation*}
\begin{split}
\delta&=(||\vec h_1||_{S_0}^2+||\vec k_1||_{S_0}^2-||\vec h_2||_{S_0}^2-||\vec k_2||_{S_0}^2)^2\\
&+4(<\vec h_1,\vec h_2>_{S_0}+<\vec k_1,\vec k_2>_{S_0})^2
+4(<\vec k_1,\vec h_2>_{S_0}-<\vec h_1,\vec k_2>_{S_0})^2.
\end{split}
\end{equation*}
As $\nu_1+\nu_2=I(x_0)$, it is enough to compute 
\begin{equation*}
\begin{split}
\nu_1\nu_2&=f(\varphi,\theta)=\det\Sigma=
(||\vec h_1||^2_{S_0}+||\vec k_1||^2_{S_0})(||\vec h_2||^2_{S_0}+||\vec k_2||^2_{S_0})\\
&-|<\vec h_1,\vec h_2>_{S_0}+<\vec k_1,\vec k_2>_{S_0}|^2-|<\vec k_1,\vec h_2>_{S_0}-<\vec h_1,\vec k_2>_{S_0}|^2,
\end{split}
\end{equation*}
defined on the 2-sphere $SO(4)/U(2)$ endowed with the polar coordinates $\theta\in[0,2\pi]$ and $\varphi\in[-\pi/2,\pi/2]$.
An orthonormal basis $\{\vec e_1,\vec e_2,\vec e_3,\vec e_4\}$ of $\R^4$ being chosen such that 
$S_0=diag(\sigma_1,\sigma_2,\sigma_3,\sigma_4)$, the partitions $(12)(34),\, (13)(24)$ and $(14)(23)$ of $\{1,2,3,4\}$ into two equal subsets give rise to three great circles $\theta=0\mod\pi,\, \theta=\pi/2\mod\pi$ and $\varphi=0$ of adapted hermitian structures $J_{\rho, P}$, with respectively 

\begin{equation*}
\begin{split}
&P=\begin{pmatrix}
1&0&0&0\\
0&1&0&0\\
0&0&1&0\\
0&0&0&1
\end{pmatrix},\quad 
P=\begin{pmatrix}
1&0&0&0\\
0&0&-\epsilon&0\\
0&\epsilon&0&0\\
0&0&0&1
\end{pmatrix},
\quad \hbox{and}\quad 
P=\begin{pmatrix}
0&0&1&0\\
0&1&0&0\\
-1&0&0&0\\
0&0&0&1
\end{pmatrix},\\
&\rho=\begin{pmatrix}
\epsilon\cos\varphi&-\sin\varphi\\
\sin\varphi&\epsilon\cos\varphi
\end{pmatrix},\quad 
\rho=\begin{pmatrix}
\cos\varphi&-\sin\varphi\\
\sin\varphi&\cos\varphi
\end{pmatrix},\quad\hbox{and}\quad 
\rho=\begin{pmatrix}
\cos\theta&\sin\theta\\
-\sin\theta&\cos\theta
\end{pmatrix},
\end{split}
\end{equation*} 
with, in the first case, $\epsilon=+1$ (resp. $-1$) if $\theta=0$ (resp. $\pi$) and in the second, 
$\epsilon=+1$ (resp. $-1$) if $\theta=\pi/2$ (resp. $3\pi/2$). Let
$$J_0=\begin{pmatrix}
0&0&-1&0\\
0&0&0&-1\\
1&0&0&0\\
0&1&0&0
\end{pmatrix},\;
J_1=\begin{pmatrix}
0&1&0&0\\
-1&0&0&0\\
0&0&0&-1\\
0&0&1&0
\end{pmatrix},\;
J_2=\begin{pmatrix}
0&0&0&-1\\
0&0&1&0\\
0&-1&0&0\\
1&0&0&0
\end{pmatrix}.
$$
The intersections of the circles are $J(0,0)=J_0,\;  J(0,\pi)=-J_0,\;  J(0,3\pi/2)=J_1,\;$  $J(0,\pi/2)=-J_1,\;  J(\pi/2,\theta)=J_2,\; 
J(-\pi/2,\theta)=-J_2$. 
\medskip

Let us call {\it critical complex structures} the critical points of $f$
\begin{lemma}
If the inertial eigenvalues $\sigma_1,\sigma_2,\sigma_3,\sigma_4$ are pairwise distinct, the critical complex structures are $\pm J_0, \pm J_1, \pm J_2$, that is the ones forced by the symmetries of $f$;  the corresponding critical values are respectively 
$$(\sigma_1+\sigma_3)(\sigma_2+\sigma_4),\quad (\sigma_1+\sigma_2)(\sigma_3+\sigma_4),\quad  (\sigma_1+\sigma_4)(\sigma_2+\sigma_3).$$
\end{lemma}
\begin{proof}
A direct computation gives the following expressions for the derivatives of the function $f(\varphi,\theta)$:
\begin{equation*}
\begin{split}
\frac{\partial f}{\partial \varphi}
&=2\sin\varphi\cos\varphi D_\varphi,\quad\hbox{with}\\
D_\varphi&=\bigl(\sigma_1-(\sigma_2\cos^2\theta+\sigma_3\sin^2\theta)\bigr)(\sigma_2\sin^2\theta+\sigma_3\cos^2\theta-\sigma_4)\\
&+(\sigma_2-\sigma_3)^2\cos^2\theta\sin^2\theta,\\
\frac{\partial f}{\partial\theta}&=2(\sigma_2-\sigma_3)(\sigma_4-\sigma_1)\cos\theta\sin\theta\cos^2\varphi.
\end{split}
\end{equation*}
It follows that the critical points are given by the equations 
$$1)\; \cos\varphi=0\quad\quad\hbox{or}\quad\quad 2)\;  \sin\varphi=0\;\hbox{and}\; \sin\theta\cos\theta=0,$$
from which the computation of the critical $J$'s is explicit as well as the computation of the critical values.
\smallskip

Finally, one checks that the symmetry $(-1,D''_P)$ correspond in each of the three cases to the invariance  of $f(\varphi,\theta)$ under the symmetry with respect to the corresponding great circle, that is

$(\varphi,\theta)\mapsto (\varphi,-\theta)$ in the first case, 

$(\varphi,\theta)\mapsto (\varphi,\pi-\theta)$ in the second case,

$(\varphi,\theta)\mapsto (-\varphi,\theta)$ in the third case.

Each critical point is a complex structure invariant under the action of a subgroup of order 4 of $\Gamma$, generated by the symmetries with respect to two orthogonal great circles, and hence is {\it forced by the symmetries of $f$.}
This ends the proof of the lemma.
\end{proof}

\begin{corollary} The conjecture is true when $p=2$: the image of the frequency map $\mathcal{F}$ coincides with the image $\mathcal{A}$ of the adapted hermitian structures.
\end{corollary}

\noindent{\bf Remarks.} 
{\bf 1)}  As the product of the symmetries with respect to the three planes $xOy,yOz,zOx$ induces the antipody on $S^2$, the function $f$ is defined on $P_2(\R)$.
\smallskip

{\bf 2)} Let us suppose that the configuration $x_0$ is 3 dimensional and that 
$$\sigma_1>\sigma_2>\sigma_3>\sigma_4=0.$$ The critical values of $\det\Sigma$ are ordered as follows:
$$
(\sigma_2+\sigma_3)\sigma_1>
(\sigma_1+\sigma_3)\sigma_2>
(\sigma_1+\sigma_2)\sigma_3.
$$
or, as $\sigma_1+\sigma_2+\sigma_3=I(x_0)$, 
$$(I(x_0)-\sigma_1)\sigma_1>(I(x_0)-\sigma_2)\sigma_2>(I(x_0)-\sigma_3)\sigma_3.$$
The maximum value of $\nu_1\nu_2=\det\Sigma$ is attained when $\sigma_1=I(x_0)/2$ maximizes the function $(I(x_0)-\sigma)\sigma$. The frequencies $\nu_1, \nu_2$ of the angular momentum are then equal, which means that the angular momentum is conjugated to a multiple of $J_0$. 

An example of this is given by an equal mass equilateral triangle in the plane generated by $\vec e_1$ and $\vec e_2$ rotating under any complex structure of the form $J=(\sin\varphi)J_0+(\cos\varphi )J_2$, that is such that $J\vec e_1$ belongs to the plane generated by $\vec e_3$ and $\vec e_4$. 

Another example is the regular tetrahedron in $\Z/{3\Z}$ symmetrix position 
with masses $(\mu,\mu,\mu,3\mu)$
mentioned in \cite{C2}. 
\medskip

\hspace{-1.25cm}
\includegraphics[scale=1]{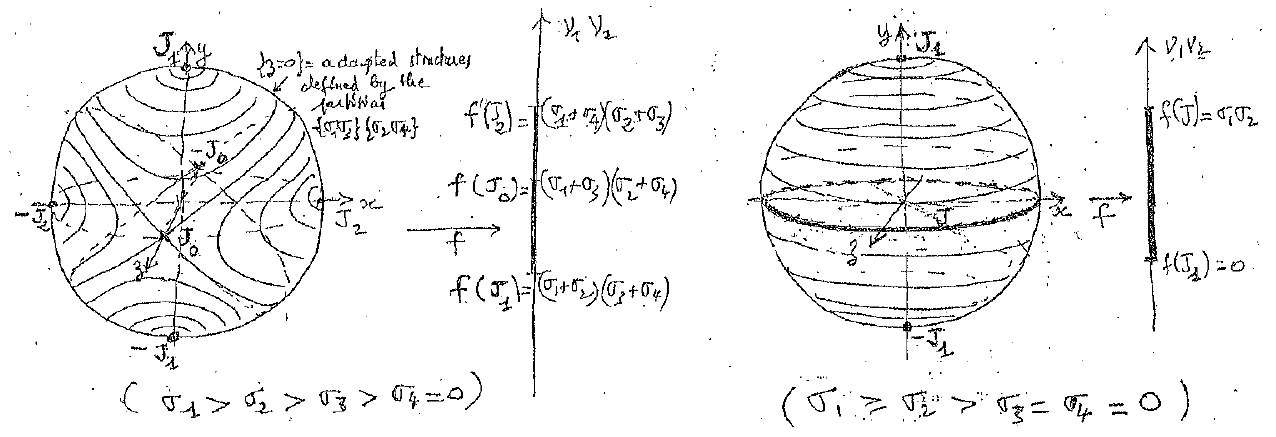}

Notice that nothing changes qualitatively in the first figure if $\sigma_4>0$ is small.
\smallskip

{\bf 3) The round cases.} 
If the ellipsoid of inertia is a round ball, i.e. if $S_0=\sigma Id$, the angular momentum of all the rotational motions in $\R^4$ will be the same up to rotation. Indeed,
in this case, $\mathcal{C}_0=2\omega\sigma J_0$.
In the same vein, we have the
\begin{lemma}
Let $x_0$ be a 3 dimensional central configuration of $N$ bodies and $\mu$ be masses such that 
$S_0=diag (\sigma,\sigma,\sigma,0)$. Whatever be the relative equilibrium motion in $\R^4$ of $x_0$ with such masses, the angular momentum endomorphism $\mathcal{C}$ has eigenvalues $\pm\frac{2i\omega}{3}I(x_0),\pm\frac{i\omega}{3}I(x_0)$.
\end{lemma}
\begin{proof} 
The proof is based on the simple remark that, under the hypotheses, if $\vec a=(a_1,a_2,a_3,a_4)$ and $\vec b=(b_1,b_2,b_3,b_4)$ are two vectors in $\R^4$, we have
$$<\vec a,\vec b>_{S_0}=\sigma(<\vec a,\vec b>_\epsilon-a_4b_4)$$
Hence the endomorphism
 $\Sigma_0$ depends only on the last column of the matrix $R$, but every element of $SO(4)/U(2)$ has a representative whose last column is $(0,0,0,1)$, hence 
 $$\Sigma_0=
\sigma\begin{pmatrix}
 2&0\\
 0&1
 \end{pmatrix},$$
 which proves the lemma.
 \goodbreak

\end{proof}
\smallskip

\noindent Examples satisfying  this lemma are the regular tetrahedron with four equal masses or sufficiently symmetric $N$-body configurations.

\section{Picture of $\mathcal{A}$ in an example with $p=3$}

The 6-dimensional space of positive hermitian structures is diffeomorphic to the complex projective space $P_3(\C)$. There are 15 basic hermitian structures $J_\pi$, corresponding to the 15 partitions $\pi$ of the set $\{1,2,3,4,5,6\}$ into 3 pairs. Each $J_\pi$ is invariant under a subgroup of order 8 of $\Gamma$ (described in \ref{Adapted}) and belongs to three circles, each one invariant under a subgroup of order 4 of $\Gamma$ and containing 4 basic hermitian structures $\pm J_\pi$ and $\pm J_{\pi'}$. Each two such circles are contained in a 3-dimensional subspace formed by the adapted structures invariant under a subgroup of $\Gamma$ of order 2. 
\smallskip

In the following figure, the 15 elements of the basic set 
$$\mathcal{B}\subset\{(x_1,y_1,z_1)\in\R^3, x_1+x_2+x_3=I(x_0),\;  x_1\ge x_2\ge x_3\ge 0\}$$
are labelled from 1 to 15 in the case $p=3$ and we have the following ordering~: 
$\sigma_1\ge\sigma_2\ge\cdots\ge\sigma_6$ and
\begin{equation*}
\begin{split}
&\sigma_1+\sigma_2>\sigma_1+\sigma_3>\sigma_1+\sigma_4>\sigma_1+\sigma_5>\sigma_1+\sigma_6>\\
&\sigma_2+\sigma_3>\sigma_2+\sigma_4>
\sigma_2+\sigma_5>\sigma_2+\sigma_6>\\
&\sigma_3+\sigma_4>\sigma_3+\sigma_5>\sigma_3+\sigma_6>\\
&\sigma_4+\sigma_5>\sigma_4+\sigma_6>\\
&\sigma_5+\sigma_6.
\end{split}
\end{equation*}

The total number of partitions of $\{1,2,\cdots,6\}$ into two subsets of 3 elements is $C_3^6=20$. Each choice selects 6 basic frequencies, which makes a total of 120. Hence, each basic frequency belongs to $120/15=8$ families. The grey polygon in the picture is $A_{\sigma_-,\sigma_+}$, where $\sigma_-,\sigma_+$ are determined by the partition 
$$\pi_0=(1,3,5)(2,4,6).$$ 
For its construction we have used the explicit inequalities (6 from Weyl, 5 from Lidskii and Wielandt and one more) which characterize this image in the case $p=3$ (see \cite{F1,B}). The three smaller polygons $A_{\pi_i},\; i=1,2,3$, contained in $A_{\sigma_-,\sigma_+}=A_{\pi_0}$, correspond to the partitions
$$\pi_1=(1,5,6)(2,3,4),\pi_2=(1,2,6)(3,4,5),\pi_3=(1,2,3)(4,5,6).$$

\vspace{0cm}\hspace{-0.5cm}
\includegraphics[scale=1]{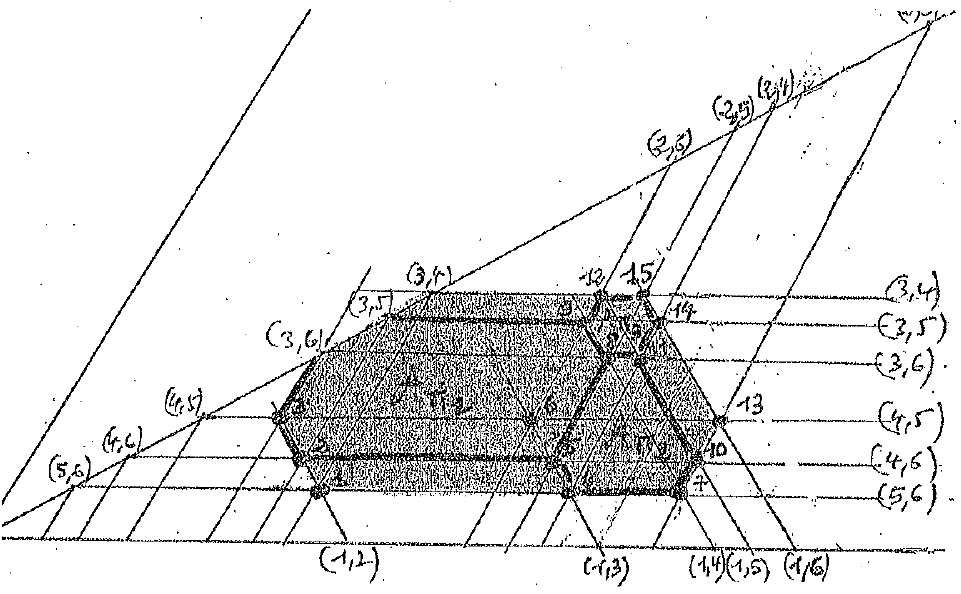}

The labels from (1,2) to (5,6) on the sides of the domain are short notations for $\sigma_1+\sigma_2$ to $\sigma_5+\sigma_6$. 
\smallskip

\section{Bifurcation values of the frequency map}

What values taken by the frequency map correspond to periodic relative equilibrium motions of $x_0$ from which stems a family of quasi-periodic relative equilibria obtained by deforming $x_0$ through balanced configurations~? 
It results from the description we gave in section 2 that, the ordered frequencies at such a bifurcation point lie in $\mathcal{A}$ and even  in the codimension one subset of the image of $\mathcal{F}$ formed by the union of the $\mathcal{A}_{k_1,k_2,\cdots,k_r}$ such that $r\ge 2$.
\smallskip

\noindent {\bf Remarks.} {\bf 1)} The above statement does not imply that the bifurcation values lie on the boundary of the image of 
$\mathcal{F}$.
\smallskip

{\bf 2)} In the generic case where each $E_i$ is of dimension 2, the angular momentum's frequencies $\nu_1,\nu_2,\cdots,\nu_p$ have no choice but being 
$$\nu_i={\omega_i}(\sigma_{2i-1}+\sigma_{2i})\quad \hbox{for}\quad  i=1,\cdots,p.$$

\noindent Symmetric balanced configurations of 3 or 4 bodies and their associated relative equilibrium motions in $\R^3$ or $\R^4$ are studied in \cite{C2}.
\smallskip

\noindent{\bf Acknowledgments} Thanks to Paul Gauduchon, Ivan Kupka, Tudor Ratiu for discussions; thanks to Terry Tao for the key reference [FFLP]; thanks to Hugo Jim\'enez-P\'erez for non-destructive (up to now) numerical testing of the conjecture for $p=3$ and a thorough reading of the manuscript which led to the elimination of many typos and ambiguities.

\vspace{0cm}\hspace{-1.5cm}
\includegraphics[scale=0.5]{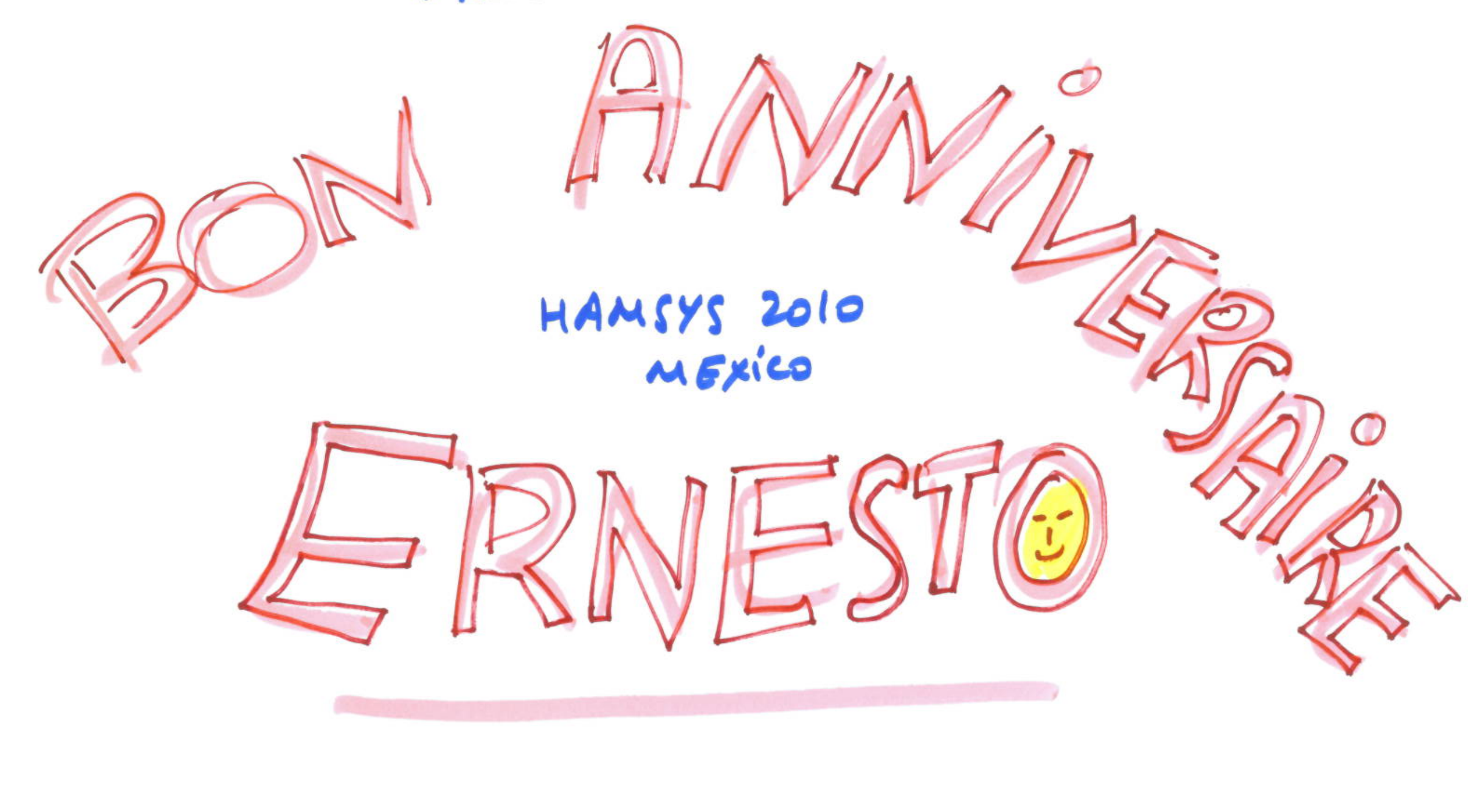}

\noindent{\bf P.S. THE CONJECTURE IS NOW PROVED:}  see \cite{CJ}

\end{document}